\documentclass{amsart}
\usepackage{amssymb}

\usepackage{color}
\usepackage{amssymb}
\usepackage{amsmath}
\usepackage{amscd}
\usepackage{amsbsy}
\usepackage{euscript}
\usepackage{verbatim}
\newtheorem{theorem}{Theorem}[section]
\newtheorem{lemma}[theorem]{Lemma}
\newtheorem{corollary}[theorem]{Corollary}
\newtheorem{proposition}[theorem]{Proposition}

\theoremstyle{definition}
\newtheorem{definition}[theorem]{Definition}

\newtheorem{example}[theorem]{Example}

\newtheorem{def-not}[theorem]{Definition and Notation}

\renewcommand{\qed}{\hfill $\square$}

\definecolor{darkgreen}{rgb}{0.03, 0.5, 0.03}

\newcommand{\mx}{\rm{Max}}

\newcommand{\lra}{\Leftrightarrow}

\newcommand{\st}{\star}

\newcommand{\ra}{\Rightarrow}
\newcommand{\mb}{\mathbb}

\newcommand{\bs}{\bigskip}

\newcommand{\mc}{\mathcal}

\begin{document}

\title[$\st$-power conductor domains]{On $\st$-power conductor domains}

\author{D.D. Anderson}
\address{Department of Mathematics, The University of Iowa, Iowa City, IA 52242 U.S.A.}
\email{dan-anderson@uiowa.edu}

\author{Evan Houston}
\address{Department of Mathematics and Statistics, University of North Carolina at Charlotte,
Charlotte, NC 28223 U.S.A.}
\email{eghousto@uncc.edu}

\author{Muhammad Zafrullah}
\address{Department of Mathematic, Idaho State University, Pocatello, ID 83209 U.S.A.}
\email{mzafrullah@usa.net}


\thanks{The second-named author was supported by a grant from the Simons Foundation (\#354565)}

\date{\today}

\begin{abstract} Let $D$ be an integral domain and $\st$ a star operation defined on $D$.  We say that $D$ is a $\st $-power conductor domain ($\st $-PCD) if for each pair $a,b\in D\backslash (0)$ and for each
positive integer $n$ we have $Da^n\cap Db^n=((Da\cap Db)^{n})^{\st }.$
We study $\st$-PCDs and characterize them as root closed domains satisfying $((a,b)^{n})^{-1}=(((a,b)^{-1})^{n})^{
\st}$ for all nonzero $a,b$ and all natural numbers $n\ge 1$.  From this it follows easily that Pr\"ufer domains are  $d$-PCDs (where $d$ denotes the trivial star operation), and $v$-domains (e.g., Krull domains) are $v$-PCDs. 
We also consider when a $\st$-PCD is completely integrally closed, and this leads to new characterizations of Krulll domains.  In particular, we show that a Noetherian domain is a Krull domain if and only if it is a $w$-PCD.
\end{abstract}

\maketitle


\section*{introduction}

Let $D$ be an integral domain with quotient field $K$.  
For $a,b \in D \setminus (0)$ and $n$ a positive integer, it is clear that $Da^n \cap Db^n \supseteq (Da \cap Db)^n$, and it is elementary that we have equality (for all $a,b,n$) if $D$ is a GCD-domain (e.g., a UFD) or a Pr\"ufer domain.  
On the other hand, Krull domains, even integrally closed Noetherian domains, may allow $Da^n \cap Db^n  \supsetneq (Da \cap Db)^n$ for some nonzero $a,b$ and $n >1$ (see \cite[Section 3]{aaj}).  However, recalling the $v$-operation on the domain $D$, given by $A_v=(D:(D:A))$ for nonzero fractional ideals $A$ of $D$ and letting $D$ be a Krull domain, we do have $Da^n \cap Db^n=((Da \cap Db)^n)_v$ for all nonzero $a,b \in D$ and $n\ge 1$.  

Now the $v$-operation is an example of a star operation.  We recall the definition:  Denoting by $\mc F(D)$ the set of nonzero fractional ideals of $D$, a map $\st:\mc F(D) \to \mc F(D)$ is a \emph{star operation} if the following conditions hold for all $A, B \in \mc F(D)$  and all $c \in K \setminus (0)$: \begin{enumerate}
\item $(cA)^{\st}=cA^{\st}$ and $D^{\st}=D$; 
\item $A \subseteq A^{\st}$, and, if $A \subseteq B$, then $A^{\st} \subseteq B^{\st}$; and 
\item $A^{\st \st}=A^{\st}$.
\end{enumerate}

The most frequently used star operations (as well as the most important for our purposes) are the $d$-operation, given, for $A \in \mc F(D)$, by $A_d=A$; the $v$-operation, defined above; the $t$-operation, given by $A_t=\bigcup B_v$, where the union is taken over all nonzero finitely generated fractional subideals $B$ of $A$; and the $w$-operation, given by $A_w=\{x \in K \mid xB\subseteq A \text{ for some finitely generated ideal $B$ of $D$ with } B_v=D\}$.  For any star operation $\st$ on $D$, we have $d \le \st \le v$, in the sense that $A=A_d \subseteq A^{\st} \subseteq A_v$ for all nonzero fractional ideals $A$ of $D$.  Other basic properties of $\st$-operations may be found in \cite[Sections 32, 34]{g} (but we do review much of what we use in the sequel).


 For a star operation $\st$ on $D$, we say that $D$ is a \emph{$\st$-power conductor domain} ($\st$-PCD) if $Da^n \cap Db^n=((Da \cap Db)^n)^{\st}$ for all $a,b \in D \setminus (0)$ and all positive integers $n$.  (The reason for the ``conductor'' terminology will become clear after Definition~\ref{d:subn} and Proposition~\ref{p:pcdequiv} below.)  As mentioned above, examples of $d$-PCDs include GCD-domains and Pr\"ufer domains, while Krull domains are $v$-PCDs.  In Section~\ref{s:weak}, as a consequence of Theorem~\ref{t:lcm-stable}, we show that essential domains, i.e., domains possessing a family $\mc P$ of prime ideals with $D= \bigcap_{P \in \mc P} D_P$ with each $D_P$ a valuation domain, are $v$-PCDs.  We also show (Proposition~\ref{p:rootclsd}) that if $D$ is a $\st$-PCD, then $D$ must be root closed and that for each maximal ideal $M$ of $D$, we must have $M$ invertible, $M^{-1}=D$, or $M=(M^2)^{\st}$.  This leads to a characterization of $\st$-PCDs as root closed domains $D$ satisfying $((a,b)^n)^{-1}=(((a,b)^{-1})^n)^{\st}$ for all nonzero $a,b \in D$. This latter condition is obviously a weakened form of invertibility; indeed, as an easy corollary we obtain that so-called $\st$-Pr\"ufer domains, domains $D$ in which each nonzero finitely generated ideal $A$ satisfies $(AA^{-1})^{\st}=D$, are $\st$-PCDs.
 
 Section~\ref{s:examples} presents examples of the subtleties involved.  Among others, we give examples of $d$-PCDs that are not integrally closed and hence not essential, $v$-PCDs that are not $d$-PCDs, and $v$-PCDs with non-$v$-PCD localizations.  In Section~\ref{s:cic} we study complete integral closure in $v$-PCDs and give several new characterizations of Krull domains.  For example, we show that $D$ is a Krull domain if and only if $D$ is a $v$-PCD in which $v$-invertible ideals are $t$-invertible and $\bigcap_{n=1}^{\infty} (M^n)_v=(0)$ for each maximal $t$-ideal $M$ of $D$.  In Section~\ref{s:w} we study two notions that have appeared previously in the literature and that  are closely related to the $d$-PCD property.  We also study $w$-PCDs and show that a Noetherian domain is a Krull domain if and only if it is a $w$-PCD.  Finally we characterize Krull domains as $w$-PCDs in which maximal $t$-ideals $M$ are divisorial and satisfy $\bigcap_{n=1}^{\infty} (M^n)_w=(0)$.
 
There are (at least) two rather natural ways to weaken the $\st$-PCD notion (see Definition~\ref{d:weak}).  We show that the three notions are distinct and, where possible (and convenient), prove results in somewhat greater generality than described above.  

We use the following notational conventions: The term ``local'' requires a ring to have a unique maximal ideal but does not require it to be Noetherian; $\subset$ denotes proper inclusion; and for fractional ideals $A,B$ of $D$, $(A:B)=\{x \in K \mid xA \subseteq B\}$, while $(A:_D B)=\{d \in D \mid dA \subseteq B\}$.


\section{Weak $\st$-PCDs} \label{s:weak}

Throughout this section, $D$ denotes a domain and $K$ its quotient field. We begin with our basic definition(s).

\begin{definition} \label{d:weak} Let $\st$ be a star operation on $D$ and $n$ a positive integer. We say that a pair $a,b\in
D\setminus (0)$
\emph{satisfies $\st _{n}$} if $Da^n\cap Db^n=((Da \cap Db)^{n})^{\st}$. We then say that $D$ \begin{enumerate}
\item is a \emph{weak $\st$-PCD} if for each pair $a,b \in D \setminus (0)$ there is an integer $m>1$, depending on $a,b$, such that $a,b$ satisfies $\st_m$; 
\item \emph{satisfies $\st_n$} if each pair of nonzero elements of $D$ satisfies $\st_n$;
\item is a \emph{$\st$-PCD} if $D$ satisfies $\st_n$ for each $n \ge 1$.
\end{enumerate}
\end{definition}

With the notation above, since the ideal $Da^n \cap Db^n$ is divisorial and hence a $\st$-ideal, it is clear that the inclusion $Da^n \cap Db^n \supseteq ((Da \cap Db)^n)^{\st}$ holds automatically.  Of course, we have equality when $n=1$. 
It is also clear that for any $n>1$, $D$ is a $\st$-PCD $\ra$ $D$ satisfies $\st_n$ $\ra$ $D$ is a weak $\st$-PCD. In Example~\ref{e:pvdstuff} below, we show that these notions are distinct when $\st$ is $d$ or $v$.

For $x,y\in D\setminus \{0\}$, we have $xy(x,y)^{-1}=xy(Dx^{-1} \cap Dy^{-1})=Dx \cap Dy =Dx(D \cap D(y/x))=Dx(Dx:_D Dy)=Dx(D:_D D(y/x))$. Hence for $a,b \in D \setminus (0)$,  $u=b/a$, $\st$ a star operation on $D$, and $n \ge 1$, we have (using the fact that we may cancel nonzero principal ideals in equations involving star operations) $Da^n \cap Db^n  = ((Da \cap Db)^n)^{\st}$ $\lra$ $(Da^n :_D Db^n)=((Da:_D Db)^n)^{\st}$ $\lra$ $(D:_D Du^n)=((D:_D Du)^n)^{\st}$ $\lra$ $(a^n,b^n)^{-1}=(((a,b)^{-1})^n)^{\st}$.
  
Motivated by this, we state the next definition and proposition.

\begin{definition} \label{d:subn} Let $\st$ be a star operation on $D$ and $n$ a positive integer. We say that an element $u \in K \setminus (0)$ \emph{satisfies $\st_n$} if $(D:_D Du^n) =((D:_D Du)^n)^{\st}$ (equivalently, $(1,u^n)^{-1}=(((1,u)^{-1})^n)^{\st}$).
\end{definition}

\begin{proposition} \label{p:pcdequiv} Let $a,b \in D \setminus (0)$, $u=b/a$, and $n \ge 1$.  The following statements are equivalent: \begin{enumerate}
\item The pair $a,b$ satisfies $\st_n$.
\item The element $u$ satisfies $\st_n$.
\item $(a^n,b^n)^{-1}=(((a,b)^{-1})^n)^{\st}$. \qed
\end{enumerate}
\end{proposition}

A consequence of the next result is that if $D$ is a $\st$-PCD for any $\st$, then $D$ is a $v$-PCD.

\begin{lemma} \label{l:ge} Let $\st' \ge \st$ be star operations on $D$, $a,b \in D \setminus (0)$, and $n \ge 1$.  \begin{enumerate}
\item If $a,b$ satisfies $\st_n$, then $a,b$ also satisfies $\st'_n$. {\rm (}Equivalently, if $u \in K \setminus (0)$ satisfies $\st_n$, then $u$ also satisfies $\st'_n$.{\rm )}. 
\item  If $D$ is a weak $\st$-PCD {\rm (}satisfies $\st_n$, is a $\st$-PCD{\rm )}, then $D$ is a weak $\st'$-PCD, {\rm (}satisfies $\st'_n$, is a $\st'$-PCD{\rm )}.
\item If $D$ is a weak $\st$-PCD {\rm (}satisfies $\st_n$, is a $\st$-PCD{\rm )}, then $D$ is a weak $v$-PCD {\rm (}satisfies $v_n$, is a $v$-PCD{\rm )}.
\end{enumerate}
\end{lemma}
\begin{proof} (1) Assume that $a,b \in D \setminus (0)$ satisfies $\st_n$ for some $n$.  Then $(Da^n \cap Db^n)=((Da \cap Db)^n)^{\st} \subseteq ((Da \cap Db)^n)^{\st'} \subseteq Da^n \cap Db^n$ (since $Da^n \cap Db^n$ is divisorial and hence automatically a $\st'$-ideal). Statement (2) follows from (1), and (3) follows from (2).
\end{proof}.

\begin{proposition} \label{p:weakchar} Let $D$ be a domain, and let $a,b \in D\setminus (0)$. Then $a,b$  satisfies $\st_m$ for some $m>1$ if and only if there is a sequence $1<n_{1}<n_{2}< \cdots $ such that $a,b$ satisfies $\st _{n_i}$ for each $i$. Hence $D$ is a weak $\st$-PCD if and only if for each pair $a,b \in D \setminus (0)$, there is a sequence $1<n_{1}<n_{2}< \cdots $
such that $a,b$ satisfies $\st _{n_i}$ for each $i$.
\end{proposition}
\begin{proof} ($\Leftarrow$) Clear. 

($\Rightarrow$) We have $Da^{n_{1}} \cap
Db^{n_{1}}=((Da\cap Db)^{n_{1}})^{\st }$ for some $n_{1}>1$. Suppose 
$1<n_{1}< \cdots <n_{k}$ have been chosen so that $Da^{n_{i}}$ $\cap
Db^{n_{i}}=((Da\cap Db)^{n_{i}})^{\st }$ for $i=1,...,k$. Choose $n>1$ so
that $(Da^{n_{k}})^{n}\cap (Db^{n_{k}})^{n}=((Da^{n_{k}}\cap
Db^{n_{k}})^{n})^{\st}$. Put $n_{k+1}=n_kn$.  Then \begin{align} & Da^{n_{k+1}} \cap Db^{n_{k+1}} = (Da^{n_k})^n \cap (Db^{n_k})^n =(( Da^{n_k} \cap Db^{n_k})^n)^{\st} \notag \\  &=((((Da \cap Db)^{n_k})^{\st})^n)^{\st}= ((Da \cap Db)^{n_kn})^{\st}=((Da \cap Db)^{n_{k+1}})^{\st}, \notag \end{align}
as desired.
\end{proof}

Recall that if $D\subseteq R$ is an extension of domains, then $R$ is said to be \emph{LCM-stable  
over $D$} if $(Da\cap Db)R=Ra\cap Rb$, for each $a,b \in D$, equivalently, if $(D:_D Du)R=(R:_R Ru)$ for each $u \in K \setminus D$. (LCM-stability was introduced by R. Gilmer \cite{g2} and popularized by H. Uda \cite{u,u2}.)  Each flat overring of $D$ is LCM-stable over $D.$

Now let $\{D_{\alpha}\}_{\alpha \in \mc A}$ be a family of overrings of $D$ with $D=\bigcap_{\alpha \in \mc A} D_{\alpha}$, and for each $\alpha \in \mc A$, let $\st_{\alpha}$ be a star operation on $D_{\alpha}$.  For a nonzero  fractional ideal $I$ of $D$, set $I^{\st}=\bigcap_{\alpha \in \mc A} (ID_{\alpha})^{\st_{\alpha}}$.  Then $\st$ is a star operation on $D$ \cite{a}.

\begin{theorem} \label{t:lcm-stable} Let $\{D_{\alpha}\}_{\alpha \in \mc A}$ be a family of LCM-stable overrings of $D$ with $D=\bigcap_{\alpha \in \mc A} D_{\alpha}$, and for each $\alpha \in \mc A$, let $\st_{\alpha}$ be a star operation on $D_{\alpha}$. Let $n \ge 1$, and let $u$ be a nonzero element of $K$ such that $u$ satisfies $(\st_{\alpha})_n$ for each $\alpha$.  Then $u$ satisfies $\st_n$ {\rm (}where $\st$ is the star operation defined above{\rm )}.
In particular, $u$ satisfies $v_n$.
\end{theorem}
\begin{proof} We have \begin{align} (D:_D Du^n) &\subseteq \bigcap_{\alpha} (D_{\alpha}:_{D_{\alpha}} D_{\alpha}u^n) = \bigcap_{\alpha} ((D_{\alpha}:_{D_{\alpha}} D_\alpha u)^n)^{\st_{\alpha}} \notag\\ &= \bigcap_{\alpha} (((D:_D Du)^n)D_{\alpha})^{\st_{\alpha}} = ((D:_D Du)^n)^{\st}. \notag \end{align} The ``in particular'' statement follows from Lemma~\ref{l:ge}.
\end{proof}

Recall that the domain $D$ is said to be \emph{essential} if $D=\bigcap_{P \in \mc P} D_P$ for some family $\mc P$ of primes of $D$ with each $D_P$ a valuation domain.  Since valuation domains are $d$-PCDs and localizations are LCM-stable, the next two results are immediate.

\begin{corollary} \label{c:essential} An essential domain $D$ is a $v$-PCD. \qed
\end{corollary}

\begin{corollary} \label{c:locvpcd} Let $n \ge 1$. If $D_M$ satisfies $v_n$ for each maximal ideal $M$ of $D$, then $D$ satisfies $v_n$.  In particular, if each $D_M$ is a $v$-PCD, then $D$ is a $v$-PCD. \qed
\end{corollary}

The converses of both of these are false--see Examples~\ref{e:pvdstuff}(4) and \ref{e:nonloc} below.  But the $d$-PCD property holds if and only if it holds locally:

\begin{corollary} \label{c:loc} Let $n \ge 1$. The following statements are equivalent. 
\begin{enumerate}
\item $D$ satisfies $d_n$.
\item $D_S$ satisfies $d_n$ for each multiplicatively closed subset $S$ of $D$.
\item $D_P$ satisfies $d_n$ for each prime ideal $P$ of $D$.
\item $D_M$ satisfies $d_n$ for each maximal ideal $M$ of $D$.
\end{enumerate}
The statements remain equivalent if ``satisfies $d_n$'' is replaced by ``is a $d$-PCD.''
\end{corollary}
\begin{proof}  Assume that $D_M$ satisfies $d_n$ for each maximal ideal $M$ of $D$. By Theorem~\ref{t:lcm-stable}, $D$ satisfies $\st_n$ for the star operation given by $A^{\st}= \bigcap_{M \in \mx(D)} AD_{M}$. However, $\st=d$ in this case.  Hence (4) $\ra$ (1).  Now assume that $D$ satisfies $d_n$, let $S$ be a multiplicatively closed subset of $D$, and let $u \in K \setminus (0)$. Then $(D_S:_{D_S} D_Su^n)=(D:_D Du^n)D_S=(D:_D Du)^nD_S=(D_S:_{D_S} D_Su)^n$.  This gives (1) $\ra$ (2).  The other implications are trivial.
\end{proof}

In order to get the $v$-PCD property to pass to quotient rings, we need a finiteness condition. Recall that $D$ is said to be \emph{$v$-coherent} if $I^{-1}$ is a $v$-ideal of finite type for each nonzero finitely generated $I$ of $D.$  Obviously, Noetherian domains are $v$-coherent.  More generally, Mori domains, domains satisfying the ascending chain condition on divisorial ideals, are $v$-coherent.

\begin{corollary} \label{c:vcoh} Let $D$ be a $v$-coherent domain and $n \ge 1$. Then the following statements are
equivalent. \begin{enumerate}
\item $D$ satisfies $v_n$.
\item $D_{S}$ satisfies $v_n$ for every multiplicatively closed subset $S$ of $D$.
\item $D_{P}$ satisfies $v_n$ for every prime ideal $P$ of $D$.
\item $D_{M}$ satisfies $v_n$ for every maximal ideal $M$ of $D$.
\item $D_{M}$ satisfies $v_n$ for every maximal $t$-ideal $M$ of $D$.
\item There is a family $\mc P=\{P\}$ of prime ideals of $D$ such that $D=\cap_{P \in \mc P} D_{P}$ and $D_{P}$  satisfies $v_n$ for every $P\in \mc P$.
\end{enumerate}
The statements remain equivalent if ``satisfies $v_n$'' is replaced by ``is a $v$-PCD.''
\end{corollary}
\begin{proof}  Let $n > 1$ and $u \in K \setminus (0)$.  Assume that $u$ satisfies $v_n$, and let $S$ be a multiplicatively closed subset of $D$.  According to \cite[Lemma 2.5]{bz}, if $A$ is a $v$-ideal of finite type in the $v$-coherent domain $D$, then $A_vD_S=(AD_S)_{v_S}$ (where $v_S$ is the $v$-operation on $D_S$). Using this, we have $(D_S :_{D_S} D_Su^n)=(D:_D Du^n)D_S = ((D:_D Du)^n)_vD_S= ((D:_D Du)^nD_S)_{v_S}=((D_S:_{D_S} D_Su)^n)_{v_S}$.  The implication (1) $\ra$ (2) follows.  Implications (2) $\ra$ (3) $\ra$ (4) $\ra$ (6) and (3) $\ra$ (5) $\ra$ (6) are trivial.  Finally, (6) $\ra$ (1) by Theorem~\ref{t:lcm-stable}.
\end{proof}

In Corollary~\ref{c:loc} the existence of a family $\mc P=\{P\}$ of prime ideals of $D$ such that $D=\cap_{P \in \mc P} D_{P}$ and $D_{P}$ is a $d$-PCD does not suffice to conclude that $D$ is a $d$-PCD:  A Krull domain $D$ obviously has the property that $D_P$ is a $d$-PCD for each maximal $t$-ideal $P$ of $D$, but such a $D$ is a $d$-PCD if and only if each $P^n$ is divisorial, and a Krull domain need not have this property (\cite[comment following Lemma 3.7]{aaj}).  (We revisit this in Section 4 below.)

We shall make frequent use of the next result.  For rings $R \subseteq S$ and a positive integer $n$, we say that $R$ is \emph{$n$-root closed in $S$} if $u \in S \setminus R$ implies $u^n \notin R$, and we say that the domain $D$ is $n$-root closed if $D$ is $n$-root closed in $K$.

\begin{proposition} \label{p:rootclsd} Let $D$ be a domain, and let $\st$ be a star operation on $D$. \begin{enumerate}
\item If $u \in K \setminus D$ satisfies $\st_n$, then $u^n \notin D$.
\item If $D$ is a weak $\st$-PCD and $M$ is a maximal ideal of $D$, then $M$ must satisfy one of the following conditions: \begin{enumerate}
\item $M$ is invertible.
\item $M^{-1}=D$.
\item $M=(M^2)^{\st}$.
\end{enumerate}
\item If $D$ satisfies $\st_n$ for some $n >1$ {\rm (}is a $\st$-PCD{\rm )}, then $D$ is $n$-root closed {\rm (}is root closed{\rm )}, and each maximal ideal of $D$ must satisfy one of the conditions above.
\end{enumerate}
\end{proposition}
\begin{proof} (1) Let $u \in K \setminus D$, and assume that $u$ satisfies $\st_n$, $n>1$.  Then $(D:_D Du^n)=((D:_D Du)^n)^{\st} \subseteq (D:_D Du)^{\st}=(D:_D Du) \subset D$. (The last equality follows since $(D:_D Du)$ is divisorial and hence a $\st$-ideal.)

(2) Assume that $D$ is a weak $\st$-PCD, let $M$ be a maximal ideal of $D$, and assume that $M^{-1} \ne D$ and that $M$ is not invertible.  We may then find $u \in M^{-1} \setminus D$, and we have $M=(D:_D Du)$.  By hypothesis, $u$ satisfies $\st_n$ for some $n>1$.  Since $M$ is not invertible, we must have $Mu \subseteq M$, whence $Mu^n \subseteq M$.  By (1), $u^n \notin D$, and hence $M=(D:_D Du^n)=((D:_D Du)^n)^{\st}=(M^n)^{\st}$.  It then follows rather easily that $M=(M^2)^{\st}$: $$M=(M^n)^{\st} \subseteq (M^2)^{\st} \subseteq M^{\st}=((M^n)^{\st})^{\st}=(M^n)^{\st}=M.$$

(3) This follows from (1) and (2).
\end{proof}

Next, we characterize ($n$-)root closed domains.

\begin{proposition} \label{p:rootclsdchar} For $n \ge 1$, a domain $D$ is $n$-root closed if and only if $(1,u^n)^{-1}=((1,u)^n)^{-1}$ for all $u \in K$ {\rm (}equivalently, $(a^n,b^n)^{-1}=((a,b)^n)^{-1}$ for all nonzero $a,b \in D${\rm )}.
\end{proposition}
\begin{proof} Assume that $D$ is an $n$-root closed domain, and let $u \in K$.  Since $(1,u^n) \subseteq (1,u)^n$, we have $(1,u^n)^{-1} \supseteq ((1,u)^n)^{-1}$.  Let $r \in (1,u^n)^{-1}$, that is, let $r,ru^n \in D$.  Then for $1 \le k \le n$, we have $(ru^k)^n=r^{n-k}(ru^n)^k \in D$, and, since $D$ is $n$-root closed, we obtain $ru^k \in D$.  Hence $r \in ((1,u)^n)^{-1}$, as desired.

Conversely, assume $(1,u^n)^{-1}=((1,u)^n)^{-1}$ for each $u \in K$.  Then if $u^n \in D$, the left side of the equation is equal to $D$, and the right side then puts $u \in D$.
\end{proof}

For $u \in K$, we always have $(1,u^n)^{-1} \supseteq ((1,u)^n)^{-1} \supseteq (((1,u)^{-1})^n)^{\st}$ for any star operation $\st$.  It follows that $u$ satisfies $\st_n$ if and only if both inclusions are equalities. 
If we combine this observation with Proposition~\ref{p:rootclsdchar}, we obtain the following characterization of the $\st$-PCD property.

\begin{theorem} \label{t:pcdchar} A domain $D$ satisfies $\st_n$ if and only if it is $n$-root closed and $((1,u)^n)^{-1}=(((1,u)^{-1})^n)^{\st}$ for each $u \in K$. Hence $D$ is a $\st$-PCD if and only if $D$ is root closed and $((1,u)^n)^{-1}=(((1,u)^{-1})^n)^{\st}$ for each $u \in K$ and $n \ge 1$. {\rm (}Equivalently, $D$ is a $\st$-PCD if and only if $D$ is root closed and $((a,b)^n)^{-1}=(((a,b)^{-1})^n)^{\st}$ for all nonzero $a,b \in D$ and $n \ge 1$.{\rm )}
\end{theorem}
\begin{proof} If $D$ is $n$-root closed and satisfies the given equality, then $(D:_D Du^n)=(1,u^n)^{-1}=((1,u)^n)^{-1}=(((1,u)^{-1})^n)^{\st}=((D:_D Du)^n)^{\st}$.  Conversely, if $D$ satisfies $\st_n$, then $D$ is $n$-root closed, and $((1,u)^n)^{-1}=(1,u^n)^{-1}=(D:_D Du^n)=((D:_D Du)^n)^{\st}=(((1,u)^{-1})^n)^{\st}$.
\end{proof}

Let $\st$ be a star operation on $D$, and let $A$ be a nonzero fractional ideal of $D$. Then $A$ is said to be \emph{$\st$-invertible} if $(AA^{-1})^{\st}=D$.  It is well-known (and easy to show) that if $A$ is $\st$-invertible and $(AB)^{\st}=D$ for some fractional ideal $B$, then we must have $B^{\st}=A^{-1}$; furthermore, if $n$ is a positive integer and we apply this fact to the equation $(A^n(A^{-1})^n)^{\st}=D$, we also have $(A^n)^{-1}=((A^{-1})^n)^{\st}$.  We use this equality in the following theorem.  (Note that while the equality holds for $\st$-invertible ideals, it fails in general as we point out in Example~\ref{e:nonkrull} below.)

\begin{theorem} \label{t:*-inv} Let $\st$ be a star operation on $D$, let $u \in K \setminus (0)$, and assume that the fractional ideal $(1,u)$ is $\st$-invertible.  Then $u$ satisfies $\st_n$ for each $n \ge 1$. {\rm (}Equivalently, if $a,b \in D \setminus (0)$ are such that $(a,b)$ is $\st$-invertible, then the pair $a,b$ satisfies $\st_n$ for each $n \ge 1$.{\rm )}
\end{theorem}
\begin{proof} Begin with the equality $(1,u)^{2n}=(1,u)^n(1,u^n)$.  Multiplying by $((1,u)^{-1})^n$ and taking $\st$'s yields $((1,u)^n)^{\st}=(1,u^n)^{\st}$.  Taking inverses then yields $((1,u)^n)^{-1}=(1,u^n)^{-1}$, and combining this with the above-mentioned equality (with $A=(1,u)$), we have $(1,u^n)^{-1}=(((1,u)^{-1})^n)^{\st}$.  This latter equation is equivalent to ``$u$ satisfies $\st_n$.''
\end{proof}

We have the following corollary to (the proof of) Theorem~\ref{t:*-inv}.

\begin{corollary} \label{c:*-inv} Let $\st$ be a star operation on $D$, let $u \in K \setminus (0)$, and assume that the fractional ideal $(1,u)$ is $\st$-invertible.  Then $(1,u^n)^{\st}=((1,u)^n)^{\st}$.  {\rm (} Equivalently, if $a,b \in D \setminus (0)$ are such that $(a,b)$ is $\st$-invertible, then $(a^n,b^n)^{\st}=((a,b)^n)^{\st}${\rm )}.
\end{corollary}

If each nonzero finitely generated ideal of $D$ is $\st$-invertible, then $D$ is said to be a \emph{$\st$-Pr\"ufer domain} \cite{aafz}.  (Thus a $v$-Pr\"ufer domain is a $v$-domain.)    
Thus the next result is immediate from Theorem~\ref{t:*-inv}.

\begin{corollary} \label{c:starprufer}  A $\st$-Pr\"ufer domain is a $\st$-PCD. In particular, a $v$-domain is a $v$-PCD. \qed
\end{corollary}

It is well known that essential domains are $v$-domains; hence Corollary~\ref{c:starprufer} strengthens Corollary~\ref{c:essential}. Since a $v$-PCD need not be integrally closed (see Example~\ref{e:pvdstuff} below), the converse of Corollary~\ref{c:starprufer} is false.

Recall \cite[Section 32]{g} that to any star operation $\st$ on $D$, we may associate a star operation $\st_f$ given by $A^{\st_f}=\bigcup B^{\st}$, where the union is taken over all nonzero finitely generated subideals $B$ of $A$.  If $\st$ is a star operation on $D$ and $\st=\st_f$ (i.e., $\st$ is of \emph{finite type}), it is well known that (1) each nonzero element $a$ of $D$ is contained in a maximal $\st$-ideal, (2) $D=\bigcap D_P$, where the intersection is taken over all maximal $\st$-ideals $P$ of $D$, and (3) primes minimal over a nonzero element are $\st$-ideals.  When $\st=v$, $\st_f$ is the well-studied $t$-operation.  Finally, recall that a \emph{Pr\"ufer $\st$-multiplication domain} (P$\st$MD) is a domain in which each nonzero finitely generated ideal is $\st_f$-invertible.  Put another way, a P$\st$MD is a $\st$-Pr\"ufer domain $D$ in which $A^{-1}$ is a finite-type $\st$-ideal for each nonzero finitely generated ideal $A$ of $D$. (A $\st$-ideal $I$ has \emph{finite type} if $I=J^{\st}$ for some finitely generated subideal $J$ of $I$.)

\begin{corollary} \label{c:pvmd} A P$\st$MD is a $\st$-PCD.  In particular, P$v$MDs are $v$-PCDs. \qed
\end{corollary}

A domain $D$ is called an \emph{almost GCD-domain} (AGCD-domain) if for each pair $a,b \in D \setminus (0)$ there is a positive integer $n$ for which $Da^n \cap Db^n$ is principal \cite{zagcd}.  We end this section by showing that within this class of domains, a $v$-PCD must be essential.

\begin{proposition} \label{p:agcd} For an AGCD domain $D$ the following are equivalent. \begin{enumerate}
\item $D$ is a P$v$MD.
\item $D$ is essential.
\item $D$ is a $v$-domain.
\item $D$ is a $v$-PCD.
\item $D$ is root closed.
\item $D$ is integrally closed.
\end{enumerate}
\end{proposition}
\begin{proof} Implications (1) $\ra$ (2) $\ra$ (3) are well known.  For the rest, (3) $\ra$ (4) by Corollary~\ref{c:starprufer}, (4) $\ra$ (5) by Proposition~\ref{p:rootclsd}, (5) $\ra$ (6) by \cite[Theorem 3.1]{zagcd}, and (6) $\ra$ (1) by \cite[Corollary 3.8]{zagcd}.
\end{proof}


\section{Pullbacks and examples} \label{s:examples}

Let $T$ be a domain, $M$ a maximal ideal of $T$, $\varphi: T \to k := T/M $ the natural projection, and $D$ a proper subring of $k$.  Then let $R = \varphi^{-1}(D)$, that is, let $R$ be the domain arising from the following pullback of canonical homomorphisms.

  $$  \begin{CD}
        R   @>>>    D\\
        @VVV        @VVV    \\
        T  @>\varphi>>   T/M = k\\
\end{CD}$$

Since $R$ and $T$ share a nonzero ideal, they have a common quotient field, which, throughout this section, will be denoted by $K$.

\begin{lemma} \label{l:field-pullback} In the diagram above, assume that $D$ is a field and that $T=(M:M)$. Let $n >1$.  Then: \begin{enumerate}
\item $R$ satisfies $d_n$ if and only if $T$ satisfies $d_n$, $M=M^2$, and $D$ is $n$-root closed in $k$.
\item  Suppose that $T$ satisfies $v_n$ locally, $(T_M:MT_M)=T_M$, $D$ is $n$-root closed in $k$, and for each nonzero $u \in K$, if $(T_M:_{T_M} T_Mu^n)$ is principal in $T_M$, then $(T_M:_{T_M} T_Mu)$ is also principal. Then $R$ satisfies $v_n$.
\end{enumerate}
\end{lemma}
\begin{proof} (1) We begin by assuming that $T$ is local with maximal ideal $M$.   
We claim that if $u \in K$ is such that $u,u^{-1} \notin T$, then $(R:_R Ru)=(T:_T Tu)$.  To verify this, suppose that $t \in T$ satisfies $tu \in T$.  Then $t \in M \subseteq R$ since $u \notin T$, and $tu \in M \subseteq R$ since $t(tu)^{-1}=u^{-1} \notin T$.  The claim follows easily. Now assume that $R$ satisfies $d_n$.  It is clear that $M$ cannot be invertible in $R$ and also that $M^{-1} \ne R$.  Hence $M=M^2$ by Proposition~\ref{p:rootclsd}. Suppose that $t \in T$ satisfies $\varphi(t)^n \in D$. Then $t^n \in R$, whence $t \in R$ and then $\varphi(t) \in D$.  Hence $D$ is $n$-root closed in $k$.  We next show that $T$ is $n$-root closed.  For this, suppose that $u \in K \setminus T$ and $u^n \in T$.  We cannot have $u^n \in R$ since $R$ is $n$-root closed (Proposition~\ref{p:rootclsd}). Hence $M=(R:_R Ru^n)=(R:_R Ru)^n \subseteq M^n=M$.  It follows that $(R:_R Ru)=M$, whence $u \in M^{-1}=(M:M)=T$.  Hence $T$ is $n$-root closed.  Finally, let $y \in K \setminus T$.  Then $y^n \notin T$. If $y^{-1} \in T$, then $(T:_T Ty^n)=T(y^{-n})=(T:_T Ty)^n$.  If $y^{-1} \notin T$, then from the claim above, we have $(T:_T Ty^n)=(R:_R Ry^n)=(R:_R Ry)^n=(T:_T Ty)^n$.  Therefore, $T$ satisfies  $d_n$.

For the converse, assume that $T$ satisfies $d_n$ with $M=M^2$ and $D$ $n$-root closed in $k$, and let $u\in K \setminus R$.  First suppose that $u \in T$.  It is easy to see that $D$ $n$-root closed in $k$ implies that $R$ is $n$-root closed in $T$ and hence that $u^n \notin R$.  We then have  $(R:_R Ru^n)=M=M^n=(R:_R Ru)^n$.  Now suppose that $u \notin T$.  If $u^{-1} \in R$, then $(R:_R Ru^n)=Ru^{-n}=(R:_R Ru)^n$, as desired.  If $u^{-1} \notin R$, then $u^{-1} \notin T$ (lest $u^{-1} \in M \subseteq R$).  In this case (again using the claim above), we have $(R:_R Ru^n)=(T:_T Tu^n)=(T:_T Tu)^n=(R:_R Ru)^n$. Hence $R$ satisfies $d_n$.  This proves (1) in local case.

For the general case, note that each maximal ideal of $R$ is of the form $N \cap R$, where $N$ is a maximal ideal of $T$, and, for $N \ne M$, $R_{N \cap R}=T_N$ (see, e.g., \cite[Theorem 1.9]{gh}).  Localizing the diagram at $M$ yields that $R_M$ satisfies $d_n$ if and only if $T_M$ satisfies $d_n$, $M=M^2$, and $D$ is $n$-root closed in $k$.  The general case now follows easily from Corollary~\ref{c:loc}.

For (2), note that $(R:M^2)=((R:M):M)=(T:M)=T$, whence $(M^2)_v=T^{-1}=M$. Also, as above, $R$ is $n$-root closed in $T$. Now suppose that $T$ is local, and let $u \in K \setminus R$. If $u \in T$, then $u^n \in T \setminus R$ and $(R:_R Ru^n) = M=(M^n)_v= ((R:_R Ru)^n)_v$.  Suppose that $u \notin T$.  If $u^{-1} \in R$, proceed as  before.  Assume $u^{-1} \notin R$ and hence (see above) that $u^{-1} \notin T$.  Even so, it is possible that $(T:_T Tu)=Tx$ for some $x \in T$. In this case, we have $(R:_R Ru^n)=(T:_T Tu^n)=((T:_T Tu)^n)_{v_T}=(Tx)^n=(T:_T Tu)^n=(R:_R Ru)^n$.  Finally, suppose that $(T:_T Tu)$ is not principal. At this point, it is helpful to observe that if $A$ is a non-principal fractional ideal of $T$, then ($A$ is a fractional ideal of $R$ and) $A^{-1}=(M:A)=(T:A)$.  In particular, $(R:_R Ru)^{-1} =(T:(T:_T Tu))$.  By hypothesis, we have $(T:_T Tu^n)$ non-principal, and, as before, $(R:_R Ru^n)=(T:_T Tu^n)$.  By the observation, this yields $(R:_R Ru^n)^{-1}=(T:(T:_T Tu^n))=(T:(T:_T Tu)^n)$.  Then, since $(T:_T Tu)^n$ is not principal and is equal to $(R:_R Ru)^n$, we have $(R:_R Ru^n)^{-1}=((R:_R Ru)^n)^{-1}$, whence $(R:_R Ru^n)=(R:_R Ru^n)_v=((R:_R Ru)^n)_v$, as desired. This completes the proof of the local case.  An easy localization argument, together with Corollary~\ref{c:locvpcd}, then yields the general case.
\end{proof}

Consider the generic pullback diagram above, and assume that $k$ is the quotient field of $D$.  An easy calculation, or an appeal to \cite[Proposition 1.8]{fg}, yields the following facts: For $t \in T$, $(D:_D D\varphi(t)) = \varphi(R:_R Rt)$ and $\varphi^{-1}(D:_D D\varphi(t))=(R:_R Rt)$.  We use these in the next result.

\begin{theorem} \label{t:pullback}  In the diagram above, assume that $T=(M:M)$, and let $n>1$.  Then $R$ satisfies $d_n$ if and only if $T,D$ both satisfy $d_n$, $D$ is $n$-root closed in $k$, and at least one of the following holds: $M=M^2$ or $k$ is the quotient field of $D$.
\end{theorem}
\begin{proof} The case where $D$ is a field is handled by Lemma~\ref{l:field-pullback}.  Suppose that $D$ is not a field but that $k$ is the quotient field of $D$, and assume that $R$ satisfies $d_n$.  Then each localization of $T$ at a maximal ideal agrees with a localization of $R$, and hence $T$ satisfies $d_n$ by Corollary~\ref{c:loc}.  Now let $t \in T \setminus R$.  From the remarks above, we have $(D:_D D\varphi(t)^n)=\varphi(R:_R Rt^n)=\varphi((R:_R Rt)^n)=(D:_D D\varphi(t))^n$. Hence $D$ satisfies $d_n$.  Note that $D$ is automatically $n$-root closed in $k$ since it satisfies $d_n$.   For the converse, suppose that $P$ is a maximal ideal of $R$.  If $P$ is incomparable to $M$, then $P=N \cap R$ for some maximal ideal $N$ of $T$, and hence $R_P=T_N$ satisfies $d_n$.  If $P \supseteq M$, then, localization produces the following pullback diagram:  $$  \begin{CD}
        R_P   @>>>    D_{\varphi(P)}\\
        @VVV        @VVV    \\
        T_M  @>\varphi>>   T/M = k\\
\end{CD}$$

At this point, for the remainder of this part of the proof, we change notation and assume that $D$ and $T$ are local with maximal ideals $\varphi(P)$ and $M$, respectively.  Let $u \in K$.  If $u \in T$, then $(R:_R Ru^n)=\varphi^{-1}(D:_D D\varphi(u)^n)=\varphi^{-1}((D:_D \varphi(u))^n)=(R:_R Ru)^n$, as required.  If $u \notin T$ but $u^{-1} \in R$, then $(R:_R Ru^n)=Ru^{-n}=(R:_R Ru)^n$.  If $u^{-1} \notin R$, then $u^{-1} \notin T$, and it is easy to see that $(R:_R Ru^n)=(T:_T Tu^n)=(T:_T Tu)^n=(R:_R Ru)^n$.  Therefore, $R$ satisfies $d_n$.

Finally (and switching back to the original notation), assume that $D$ is not a field and that the quotient field $F$ of $D$ is properly contained in $k$. Let $S:=\varphi^{-1}(F)$.  If $R$ satisfies $d_n$, then by what was proved in the preceding paragraph, $S$ satisfies $d_n$  and (hence) $D$ is $n$-root closed in $F$.  Lemma~\ref{l:field-pullback} then yields that $T$ satisfies $d_n$, $M=M^2$, and that $F$ is $n$-root closed in $k$; it follows that $D$ is $n$-root closed in $k$.  For the converse, assume that $D$ and $T$ satisfy $d_n$, $D$ is $n$-root closed in $k$, and $M=M^2$.  To see that $F$ is $n$-root closed in $k$, let $x \in k$ with $x^n \in F$.  Write $x^n=d/e$ with $d,e \in D$.  Then $(ex)^n = e^{n-1}ex^n \in D$, whence $ex \in D$, and we have $x \in F$, as desired.  Lemma~\ref{l:field-pullback} then ensures that $S$ satisfies $d_n$, and then the preceding paragraph yields that $R$ satisfies $d_n$.
\end{proof}

Recall that a local domain $(R,M)$ is a \emph{pseudo-valuation domain} (PVD) if $M^{-1}$ is a valuation domain with maximal ideal $M$ \cite{hh}; $V$ is then called the canonical valuation overring of $R$.  It follows that a domain $R$ is a PVD if and only if it is a pullback of the type in Lemma~\ref{l:field-pullback} with $T$ a valuation domain \cite[Proposition 2.6]{ad}. We specialize Lemma~\ref{l:field-pullback} to PVDs:

\begin{corollary} \label{c:pvd} Let $(R,M)$ be a PVD with canonical valuation overring $
V=M^{-1}$, and assume that $R\subsetneq V$. Then: \begin{enumerate}
\item The following statements are equivalent. \begin{enumerate}
\item $R$ is a weak $v$-PCD.
\item $M=M^2$.
\item $R$ is a weak $d$-PCD.
\end{enumerate}
\item  If $n>1$, the following statements are equivalent. \begin{enumerate}
\item $R$ satisfies $v_n$.
\item $R$ is $n$-root closed {\rm (}equivalently, $R/M$ is $n$-root closed in $V/M${\rm )} and $M=M^{2}$.
\item $R$ satisfies $d_n$.
\end{enumerate}
\item The following statements are equivalent. \begin{enumerate}
\item $R$ is a $v$-PCD.
\item $R$ is root closed and $M=M^2$.
\item $R$ is a $d$-PCD.
\end{enumerate}
\end{enumerate}
\end{corollary}
\begin{proof}  (1) Assume that $R$ is a weak $v$-PCD, and choose $x \in V \setminus R$.  Then $(R:_R Rx)=M$.  For some integer $k$, we must have $(R:_R Rx^k)= ((R:_R Rx)^k)_v$.  If $x^k \in R$, this equality becomes $R=(M^k)_v$, which is impossible since $M$ is divisorial in $R$.  Hence $x^k \notin R$, in which case the equality above becomes $M=(M^k)_v$.  This then yields $M=(M^2)_v$, whence $V=M^{-1}=(M^2)^{-1}=((R:M):M)=(V:M)$.  It follows that $M$ cannot be principal in the valuation domain $V$, whence $M=M^2$.  Thus (a) $\ra$ (b).  Now assume that $M=M^2$, and let $y \in K \setminus R$.  If $y^2, y^3 \in R$, then $(y^2)y=y^3 \in R$, whence $y^2 \in M$.  However, since $y \in V$, this puts $y \in M \subseteq R$, a contradiction.  Therefore, for $m=2$ or $m=3$, $y^m \notin R$, whence $(R:_R Ry^m)=M=M^m=(R:_R Ry)^m$, as desired.  Finally, suppose $y \notin V$.  Then $y^{-1} \in M \subseteq R$, whence for all $s > 1$, we have $(R:_R Ry^s)=Ry^{-s}=(R:_R Ry)^s$.  This gives (b) $\ra$ (c), and (c) $\ra$ (a) follows from Lemma~\ref{l:ge}.

(2) Assume that $R$ satisfies $v_n$.  Then $R$ is $n$-root closed by Proposition~\ref{p:rootclsd}, and $M=M^2$ by (1).  The implication (b) $\ra$ (c) follows from Lemma~\ref{l:field-pullback}(1), and (c) $\ra$ (a) is trivial (Lemma~\ref{l:ge}).

(3) This follows from (2).
\end{proof}

Next, we present examples, several of which were promised above.  We begin with PVD examples, where the conclusions are immediate from Corollary~\ref{c:pvd}.

\begin{example} \label{e:pvdstuff} Let $(R,M)$ be a PVD with canonical valuation overring $V$.  Then:
\begin{enumerate}
\item If $M\ne M^2$, then $R$ is not a weak $v$-PCD (Corollary~\ref{c:pvd}). For example, take $R=F+xk[x]_{xk[x]}$, where $F \subset k$ are fields and $x$ is an indeterminate; if, in addition, $[k:F]<\infty$, then $R$ is Noetherian.
\item If $M=M^2$, but $R$ is not $n$-root closed for any $n>1$, then $R$ is a weak $d$-PCD but does not satisfy $v_n$ for any $n>1$. (For example, let $k$ be an algebraic closure of $\mb Q$, let $V=k+M$ be a non-discrete rank-one valuation domain with maximal ideal $M$, and let $R=\mb Q+M$.)
\item If $M=M^2$ and $R$ is 2-root closed but not 3-root closed (e.g., take $V=\mb F_4+M$ to be a rank-one non-discrete valuation domain with maximal ideal $M$, and let $R=\mb F_2 +M$), then $R$ satisfies $d_2$ but is not a $v$-PCD.
\item If $M=M^2$ and $R$ is root closed, then $R$ is a $d$-PCD. \begin{enumerate}
\item If $R/M$ is not algebraically closed in $V/M$, then $R$ is not integrally closed.  (For example, take $R/M=\mb Q$ and $V/M=\mb Q[u]$, where $u$ is a root of $x^3-3x+1$.)  Hence a $d$-PCD need not be integrally closed and hence need not be an essential domain (or even a $v$-domain).
\item If $R/M$ is algebraically closed in $V/M$, then $R$ is a $d$-PCD that is integrally closed but not completely integrally closed.
\end{enumerate}
\end{enumerate}
\end{example}

Let $F \subset k$ be fields, $X$ a set of indeterminates, $M$ the maximal ideal of $k[X]$ generated by $X$, and put $R=F+Xk[X]$.  Such rings have often been used to provide interesting examples.  We investigate PCD-properties in these rings.

\begin{example} \label{e:poly} With the notation above, assume that $F$ root closed in $k$.  \begin{enumerate}
\item Let $|X|=1$. Then: \begin{enumerate}
\item $R$ is not a $d$-PCD by Lemma~\ref{l:field-pullback}.  In fact, observe that in this case $R_M$ is a PVD that is not a weak $v$-PCD by Example~\ref{e:pvdstuff}.  Then, since $R$ is $v$-coherent (see, e.g. \cite[Theorem 3.5]{gh}), Corollary~\ref{c:vcoh} ensures that $R$ is not even a weak $v$-PCD.
\item If $F$ is algebraically closed in $k$, then $R$ is an integrally closed domain that is not a weak $v$-PCD.
\end{enumerate}
\item If $|X|>1$, then $k[X]$ satisfies the hypotheses of Lemma~\ref{l:field-pullback}(2).  Hence $R$ is a $v$-PCD but not a weak $d$-PCD by Proposition~\ref{p:rootclsd}.
\item If in (2) we take $1<|X|<\infty$ and $[k:F]<\infty$, then $R$ is a Noetherian $v$-PCD that is not a weak $d$-PCD.
\end{enumerate}
\end{example}

Next, we give an example showing that the $v$-PCD property does not localize.

\begin{example} \label{e:nonloc} Let $R$ be the example given by Heinzer \cite{h}.  The domain $R$ is essential, and therefore a $v$-PCD, but contains a prime ideal $P$ such that $R_P$ is not essential.  In fact, it is easy to see that $R_P$ is a PVD with $P \ne P^2$.  Hence $R_P$ is not a (weak) $v$-PCD by Corollary~\ref{c:pvd}(1).
\end{example}


\section{Completely integrally closed $v$-PCDs} \label{s:cic}

In this and the next section we return to our convention that $D$ is a domain with quotient field $K$.  Recall that $D$ is \emph{completely integrally closed} if whenever $u \in K$ and $a$ is a nonzero element of $D$ with $au^n\in D$ for all $n \ge 1$, then $u \in D$. Thus the domain $D$ is completely integrally closed if and only if $\bigcap_{n=1}^{\infty} (D:_D Du^n)=(0)$ for each $u \in K \setminus D$. It is well-known that $D$ is completely integrally closed if and only if each nonzero ideal of $D$ is $v$-invertible.  We begin with a characterization of completely integrally closed $\st$-PCDs.

\begin{proposition} \label{p:weakcap} Let $D$ be a weak $\st$-PCD. Then \begin{enumerate}
\item $\bigcap_{n=1}^{\infty} (D:_D Du^{n})=\bigcap_{n=1}^{\infty} ((D:_D Du)^{n})^{\st }$ for each $u \in K \setminus (0)$.
\item $D$ is completely integrally closed if and only if $\bigcap_{n=1}^{\infty} ((D:_D Du)^n)^{\st} =(0)$ for each $u \in K \setminus D$.
\end{enumerate}
\end{proposition}
\begin{proof} (1) Let $u \in K \setminus (0)$, and use Proposition~\ref{p:weakchar} to choose $1<n_{1}<n_{2} \cdots$ with $(D:_D Du^{n_{i}})=((D:_D Du)^{n_{i}})^{\st }$ for each $i$.  Then 
$$\bigcap_{n=1}^{\infty }(D:_D Du^{n}) \subseteq \bigcap_{i=1}^{\infty}(D:_D Du^{n_{i}})=\bigcap_{i=1}^{\infty }((D:_D Du)^{n_{i}})^{\st} =\bigcap_{n=1}^{\infty }((D:_D Du)^{n})^{\st},$$ 
and (1) follows easily.

(2) By definition $D$ is completely integrally closed if and only if $\bigcap_{n=1}^{\infty} (D:_D Du^n)=(0)$ for each $u \in K \setminus D$.  Hence the conclusion follows from (1).
\end{proof}

Combining the proposition with Corollary~\ref{c:starprufer}, we have:

\begin{corollary} \label{c:starprufercic} A $\st$-Pr\"ufer domain $D$ is completely integrally closed if and only if $\bigcap_{n=1}^{\infty} ((D:_D Du)^n)^{\st} =(0)$ for each $u \in K \setminus D$. In particular, a $v$-domain (and hence an essential domain or a P$v$MD) is completely integrally closed if and only if $\bigcap_{n=1}^{\infty} ((D:_D Du)^n)_v =(0)$ for each $u \in K \setminus D$. \qed
\end{corollary}

\begin{proposition} \label{p:maxt} A weak $\st$-PCD $D$ is completely integrally closed if for every maximal $\st_f$-ideal $P$ of $D$ we have $\bigcap_{n=1}^{\infty} (P^n)^{\st}=(0)$.  In particular, a weak $d$-PCD {\rm (}weak $v$-PCD{\rm )} $D$ is completely integrally closed if for every maximal ideal 
{\rm (}maximal $t$-ideal{\rm )} $P$ of $D$ we have $\bigcap_{n=1}^{\infty} P^n=(0)$ {\rm (}$\bigcap_{n=1}^{\infty} (P^{n})_{v}=(0){\rm )}.$
\end{proposition}
\begin{proof} Let $D$ be a weak $\st$-PCD, and let  $u \in K \setminus D$.  Then $(D:_D Du) \subseteq P$ for some maximal $\st_f$-ideal $P$ of $D$, and hence $((D:_D Du)^n)^{\st} \subseteq (P^n)^{\st}$.  The conclusion then follows from Proposition~\ref{p:weakcap}.
\end{proof}

The condition on the maximal $t$-ideals in Proposition~\ref{p:maxt} is quite stringent.  In particular, the condition requires a maximal $t$-ideal $P$ to satisfy $P_v \ne D$. Counterexamples to the converse of Proposition~\ref{p:maxt} abound.  For example, a non-discrete
rank one valuation domain $(D,M)$ is a $v$-PCD and completely integrally closed but does not satisfy $%
\cap (M^{n})_{v}=(0)$; in fact, $(M^n)_v=D$ for each $n$.  For another example, $D=k[x,y]$, $k$ a field and $x,y$ indeterminates, is a completely integrally closed $v$-PCD, but $(x,y)_v=D$.

On the other hand, if we require even more, we can obtain interesting characterizations of Krull domains. Note that if for a maximal $t$-ideal $P$ we have $\cap (P^{n})_{v}=(0)$, then $P_v \ne D$ and hence $P_v=P$, that is, $P$ is divisorial. In \cite{gv} Glaz and Vasconselos called a domain $D$ an \emph{$H$-domain} if each ideal $I$ of $D$ with $I_v=D$ contains a finitely generated ideal $J$ with $J_v=D$.  According to \cite[Proposition 2.4]{hz}, the domain $D$ is an $H$-domain if and only if every maximal $t$-ideal of $D$ is divisorial. It is also easy to see that $D$ is an $H$-domain if and only  every $v$-invertible
ideal of $D$ is $t$-invertible. In particular, if an H-domain $D$ is a $v$-domain, 
it must be a P$v$MD. Finally, Glaz and Vasconcelos showed that a domain $D$ is a Krull domain if and
only if it is a completely integrally closed $H$-domain \cite[3.2(d)]{gv}. 
\bs

\begin{corollary} \label{c:hdom} The following are equivalent for a domain $D$. \begin{enumerate}
\item $D$ is a Krull domain.
\item $D$ is a completely integrally closed $H$-domain.
\item $D$ is an $H$-domain and a $v$-domain with $t$-dimension one.
\item $D$ is an $H$-domain and a P$v$MD with $t$-dimension one.
\item $D$ is an $H$-domain and a $v$-PCD with $\bigcap_{n=1}^{\infty} (P^n)_v=(0)$ for each maximal $t$-ideal $P$ of $D$.
\item $D$ is an integrally closed $H$-domain in which $((a,b)^n)^{-1}=(((a,b)^{-1})^n)_v$ for all $a,b \in D \setminus (0)$ and $\bigcap_{n=1}^{\infty} (P^n)_v =(0)$ for each maximal $t$-ideal $P$ of $D$.
\end{enumerate}
\end{corollary}
\begin{proof} The equivalence of (1) and (2) and the implication (3) $\ra$ (4) are discussed above. It is clear that (1) $\ra$ (3) and (6). Assume (4), and let $P$ be a minimal prime of a principal ideal.  Then $P$ is a maximal $t$-ideal and is therefore divisorial (see above).  Pick $u \in P^{-1} \setminus D$.  Then, since $D$ has $t$-dimension one, $P=(D:_D Du)$. It follows that $D$ is a Krull domain by \cite[Proposition 2.4]{hz}.  Hence (4) $\ra$ (1). Finally, we have (6) $\ra$ (5) by Theorem~\ref{t:pcdchar} and (5) $\ra$ (2) by Proposition~\ref{p:maxt}.
\end{proof}

According to Corollary~\ref{c:hdom}, if $R$ is an integrally closed non-Krull $H$-domain with $\bigcap_{n=1}^{\infty} (P^n)_v=(0)$ for each maximal $t$-ideal $P$ of $R$, then we must have $((a,b)^n)^{-1} \ne (((a, b)^{-1})^n)_v$ for some nonzero $a,b \in D$ and $n>1$.  We next give an example of this phenomenon.

\begin{example} \label{e:nonkrull} Let $T$ be a PID with a maximal ideal $P$ such that $k:=T/P$ admits a field $F$ that is algebraically closed in $k$ (e.g., $T=F(y)[x]$, $y,x$ indeterminates).  Let $\varphi:T \to k$ be the natural projection and set $R=\varphi^{-1}(F)$.  Then $R$ is an integrally closed non-Krull domain (since $R$ is not completely integrally closed).  Of course, $P$ is a divisorial ideal of $R$.  In fact, each maximal ideal of $R$ is divisorial. To see this, let $Q\ne P$ be a maximal ideal of $R$.  Then $Q=N \cap R$ for 
maximal ideal $N \ne M$ of $T$.  Write $N=Tz$.  Then $z^{-1}PQ \subseteq z^{-1}PN\subseteq P \subseteq R$, and $z^{-1}P \nsubseteq R$ (indeed, $z^{-1}P \nsubseteq T$). Hence $Q$ is divisorial.  Hence $R$ is (vacuously) an $H$-domain.  Since $Q^n\subseteq N^n=Tz^n$ and $Tz^n$ is divisorial (as an ideal of $R$), we have $\bigcap_{n=1}^{\infty} (Q^n)_v \subseteq \bigcap_{n=1}^{\infty} Tz^n=(0)$.  Write $P=Tc$.  Then $P^n=Tc^n$, which is divisorial.  Hence $\bigcap_{n=1}^{\infty} (P^n)_v=\bigcap_{n=1}^{\infty} Tc^n=(0)$.  Thus $R$ has the required properties.  It is not difficult to identify elements $a,b$ as in the preceding paragraph: let $t \in T \setminus R$.  Then $(1,t)^{-1}=(R:_R Rt)=P$, whence $(((1,t)^{-1})^n)_v=P^n$.  On the other hand, $((1,t)^n)^{-1}=P$.  Now take $a=c$ and $b=ct$.
\end{example}

Let $D$ be a Noetherian domain.  Then $D$ is certainly an $H$-domain.  Moreover, $D$ is integrally closed if and only if $D$ is a Krull domain.  In view of the equivalence (1) $\lra$ (5) of Corollary~\ref{c:hdom}, we have:

\begin{corollary} \label{c:noe} A Noetherian domain $D$ is integrally closed if and only if $D$ is a $v$-PCD and $\bigcap_{n=1}^{\infty} (M^n)_v=(0)$ for each maximal $t$-ideal $M$ of $D$. \qed
\end{corollary}

As we saw in Example~\ref{e:pvdstuff}(1), a Noetherian domain need not be a $v$-PCD.
What is more interesting here is the fact that not every Noetherian domain has
the property that for every maximal $t$-ideal $M$ we have $\cap
(M^{n})_{v}=(0).$ 
We end this section with an example of this.

\begin{example} \label{e:noe2} Let $F \subset k$ be a root closed extension of fields with $[k:F]$ finite.  Let $T=k[x,y]=k+M$, $x,y$ indeterminates and $M=(x,y)$, and let $R=F+M$.  Then $R$ is Noetherian.  It is easy to see that $M^{-1}=T$, whence $R$ is a $v$-PCD by Lemma~\ref{l:field-pullback}(2).  (But $R$ is not a weak $d$-PCD  by Proposition~\ref{p:rootclsd}(2)). By direct calculation or Corollary~\ref{c:noe}, we cannot have $\bigcap_{n=1}^{\infty} (M^n)_v =(0)$.  (Indeed, $(M^n)_v=M$ for each $n\ge 1$).
\end{example}


\section{Connections with other properties} \label{s:w}

In \cite{z} a domain $D$ was said to have the \emph{{\LARGE $\st$}-property} if for $
a_{1}, \ldots,a_{m}, b_{1},,...,b_{n}\in D\setminus (0)$ we have $(\bigcap_i 
 Da_i)(\bigcap_j Db_j)=\bigcap_{i,j} Da_{i}b_{j}$.  The authors of \cite{aaj} discussed a
special case of this, which we call here the \emph{{\LARGE $\st$}{\LARGE $\st$}-property}: $(Da \cap
Db)(Dc\cap Dd)=Dac\cap Dad\cap Dbc\cap Dbd$ for all $a,b,c,d\in D\setminus 
(0)$; and they showed that a Noetherian domain satisfying {\LARGE $\st$}{\LARGE $\st$} is
locally factorial \cite[Corollary 3.9]{aaj}. Now {\LARGE $\st$}  and {\LARGE $\st$}{\LARGE $\st$} are equivalent over a Noetherian domain, for, in this case, {\LARGE $\st$}{\LARGE $\st$} implies locally (factorial and hence) 
GCD, and it is easy to see that a locally GCD-domain is a locally {\LARGE $\st$}-domain and hence a {\LARGE $\st$}-domain \cite[Theorem 2.1]{z}. Thanks to its efficiency the {\LARGE $\st$}-property  
can be used to provide a more satisfying characterization
of integrally closed Noetherian domains than does the $v$-PCD property, as we shall see below. But the {\LARGE $\st$}-property is more potent in that it can be put to use even 
 in $v$-coherent domains. For example, it was shown in
 \cite[Corollary 1.7]{zgdd} that the {\LARGE $\st$}-property makes a $v$-coherent domain
a \emph{generalized GCD-domain} (GGCD-domain): a domain in which $aD\cap bD$ is invertible
for each pair $a,b\in D\setminus (0)$. We restate the result as the
following proposition.

\begin{proposition} \label{p:ggcd} An integral domain is a GGCD-domain if and only if $D$ is a 
$v$-coherent {\LARGE $\st$}-domain. \qed
\end{proposition}

Now recall that P$v$MDs may be characterized as $t$-locally valuation domains, that is, $D$ is a P$v$MD if and only if $D_P$ is a valuation domain for each maximal $t$-ideal $P$ of $D$. In \cite{ch}, Chang used the local version of Proposition~\ref{p:ggcd} to
characterize P$v$MDs.  Before stating Chang's result, we need some background.  We first recall the $w$-operation: for a nonzero fractional ideal $A$ of $D$, $A_w=\{x \in K \mid xB \subseteq A \text{ for some finitely ideal $B$ of $D$ with } B_v=D\}$.  It is well known that $A_w=\bigcap AD_P$, where the intersection is taken over the set of maximal $t$-ideals $P$ of $D$; moreover, we have $A_wD_P=AD_P$ for each $P$. Call $D$ a \emph{{\LARGE $\st$}($w$)-domain} if $((\bigcap_{i}Da_i)(\bigcap_{j}Db_j))_{w}=\bigcap_{i,j} Da_{i}b_{j}$ for all $a_{i},b_{j} \in D \setminus (0)$.

\begin{proposition} \label{p:chang} {\rm (\cite [Theorem 3]{ch})} An integral domain $D$ is a P$v$MD if and
only if $D$ is a $v$-coherent {\LARGE $\st$}{\rm (}$w${\rm )}-domain. \qed
\end{proposition}

In view of \cite{aaj} we can introduce the notion of a \emph{{\LARGE $\st$}{\LARGE $\st$}($w$)-domain} as
a domain $D$ such that $((Da \cap Db)(Dc\cap Dd))_{w}= Dac\cap Dad\cap
Dbc\cap Dbd$, for all $a,b,c,d\in D\setminus (0)$.    We shall show that a Noetherian domain $D$ is integrally closed $\lra$ $D$ is a $w$-PCD $\lra$ $D$ is a {\LARGE $\st$}{\LARGE $\st$}($w$)-domain $\lra$ $D$ is a  {\LARGE $\st$}($w$)-domain. This is interesting in light of \cite[comment following Lemma 3.7]{aaj}, where it is shown that a Krull domain is a $d$-PCD if and only if all powers of each maximal $t$-ideal are divisorial and that an integrally closed Noetherian domain need not have this property and hence need not be a $d$-PCD.  

In fact, we can establish the result in a more general setting.  Recall that a domain $D$ is a \emph{strong Mori domain} if it satisfies the ascending chain condition on $w$-ideals.  These domains were introduced and studied by Wang and McCasland \cite{wm, wm2}.  They are characterized as domains $D$ for which (1) $D_M$ is Noetherian for every maximal $t$-ideal $M$ of $D$ and (2) $D$ has finite $t$-character (each nonzero element $a$ of $D$ is contained in only finitely many maximal $t$-ideals of $D$) \cite[Theorem 1.9]{wm2}. It is well-known (and follows easily from (1)) that an integrally closed strong Mori domain is completely integrally closed and hence a Krull domain.

\begin{theorem} \label{t:strongmori} The following statements are equivalent for a strong Mori domain $D$. \begin{enumerate}
\item $D$ is integrally closed. 
\item $D$ is a $w$-PCD.
\item $D$ is a {\LARGE $\st$}{\LARGE $\st$}{\rm (}$w${\rm )}-domain.
\item $D$ is a {\LARGE $\st$}{\rm (}$w${\rm )}-domain.
\item $D$ is completely integrally closed {\rm (}and hence a Krull domain{\rm )}.
\end{enumerate}
\end{theorem}

As mentioned above, items (1) and (5) are equivalent. The rest of the proof is contained in the next two lemmas.  For the first, we call a local domain $(D,M)$ \emph{$t$-local} if its maximal ideal is a $t$-ideal. (Perhaps a caveat is in order here.  Localizing at a maximal $t$-ideal does not in general produce a $t$-local domain!  However, this is not an issue in the strong Mori setting: for a strong Mori domain $D$, a prime $P$ of $D$ is a $t$-ideal (equivalently, divisorial) if and only if $PD_P$ is a $t$-ideal \cite[Lemma 3.17]{k}.)

\begin{lemma} \label{l:tlocnoe} For a $t$-local Noetherian domain $(D,M)$, the 
following statements are equivalent. \begin{enumerate}
\item $D$ is a {\LARGE $\st$}-domain.
\item $D$ is a {\LARGE $\st$}{\LARGE $\st$}-domain.
\item $D$ is integrally closed.
\item $D$ is a {\rm (}rank-one discrete{\rm )} valuation domain.
\item $D$ is a $d$-PCD.
\item $D$ is a weak $d$-PCD.
\end{enumerate}
\end{lemma}
\begin{proof} Implications (1) $\Rightarrow$ (2), (4) $\ra$ (5) $\ra$ (6), and (4) $\ra$ (1) are trivial, (3) $\ra$ (4) is well known, and (2) $\ra$ (3) is essentially the proof of \cite[Corollary 3.9]{aaj}. Now assume (6). Then $M$ is divisorial, whence $M^{-1} \ne D$, and, clearly, $M \ne M^2$.  Therefore, according to Proposition~\ref{p:rootclsd}, $M$ must be invertible, and hence $D$ is a rank-one discrete valuation domain, as desired.  Thus (6) $\ra$ (4), and the proof is complete.
\end{proof}

\begin{lemma} \label{l:wpcd} A domain is a $w$-PCD 
{\rm (} a {\LARGE $\st$}{\LARGE $\st$}($w$)-domain, a {\LARGE $\st$}($w$)-domain, integrally closed{\rm )} if and only if it is $t$-locally a $d$-PCD {\rm (} a {\LARGE $\st$}{\LARGE $\st$}-domain, a {\LARGE $\st$}-domain, integrally closed{\rm )}.
\end{lemma}
\begin{proof} It is well known that $D$ is integrally closed if and only if it is $t$-locally integrally closed (and follows easily from the representation $D=\bigcap D_P$, where the intersection is taken over the set of maximal $t$-ideals $P$ of $D$). Let $D$ be a $w$-PCD, and let $M$ be a maximal $t$-ideal of $D$.  For $u \in K$ we have $(D_M:_{D_M} D_Mu^n)=(D:_D Du^n)D_M=((D:_D Du)^n)_wD_M=(D:_D Du)^nD_M=(D_M:_{D_M} D_Mu)^n$.  Hence $D_M$ is a $d$-PCD.  Now assume that $D$ is $t$-locally a $d$-PCD, and let $\mc P$ denote the set of maximal $t$-ideals of $D$.  Then for $u \in K$, $(D:_D Du^n)=(D:_D Du^n)_w=\bigcap_{P \in \mc P} (D:_D Du^n)D_P=\bigcap_{P \in \mc P} (D_P:_{D_P} D_Pu^n)=\bigcap_{P \in \mc P} (D:_D Du)^nD_P=((D:_D Du)^n)_w$. The details in the proofs of the other properties are similar.
\end{proof}

Lemma~\ref{l:tlocnoe} again shows that the {\LARGE $\st$}-property is much more potent than the $v$-PCD-property: according to the lemma, a $t$-local Noetherian domain satisfying {\LARGE $\st$} must be integrally closed, whereas, if $R$ is as in Example~\ref{e:poly}, then $R_M$ is a non-integrally closed $t$-local Noetherian $v$-PCD.  For still another example, a P$v$MD is automatically a $v$-PCD (Corollary~\ref{c:pvmd}), but a P$v$MD with the {\LARGE $\st$}-property is a GGCD-domain by Proposition~\ref{p:ggcd}.  On the other hand, the $v$-PCD property is useful in determining whether a domain is completely integrally closed.

Now let us step back and take another look at (3) of Lemma~\ref{l:tlocnoe} and ask: What
if we consider a (not necessarily Noetherian) $t$-local domain with maximal ideal $M$ 
divisorial but include the condition that $\bigcap M^{n}=(0)?$ We show that
the result would still be a discrete rank one valuation domain:

\begin{proposition} \label{p:locdiv} Let $(D,M)$ be a local $d$-PCD such that $M$ is divisorial and $\bigcap_{n=1}^{\infty} M^n=(0)$. Then $D$ is a rank-one discrete valuation domain.
\end{proposition}
\begin{proof} By Proposition~\ref{p:maxt}, $D$ is completely integrally closed.  Thus $(MM^{-1})_v=D$, and then, since $M$ is divisorial, $MM^{-1}=D$.  Hence $M$ is principal.  The condition $\bigcap M^n=(0)$ then ensures that $D$ is one-dimensional, that is, that $D$ is a rank-one discrete valuation domain.
\end{proof}

We close with yet another characterization of Krull domains.

\begin{theorem} \label{t:krullw} A domain $D$ is a Krull domain if and only if it has the folllowing properties: \begin{enumerate}
\item Each maximal $t$-ideal of $D$ is divisorial.
\item $\bigcap_{n=1}^{\infty} (M^n)_w=(0)$ for each maximal $t$-ideal $M$ of $D$.
\item $D$ is a $w$-PCD.
\end{enumerate}
\end{theorem}
\begin{proof} It is well-known that a Krull domain has the first property and that the $v$-, $t$- and $w$-operations coincide.  Properties (2) and (3) then follow from Corollary~\ref{c:hdom}.  Now assume that $D$ is a domain with the properties listed.  By (1), $R$ is an $H$-domain \cite[Proposition 2.4]{hz}.  It is well-known that each maximal $w$-ideal of $D$ is a maximal $t$-ideal (and vice versa) and that the $w$-operation is of finite type.  Hence (2) and (3), together with Proposition~\ref{p:maxt}, imply that $D$ is completely integrally closed. Therefore, $D$, being a completely integrally closed H-domain, is a Krull domain \cite[3.2(d)]{gv}.
\end{proof}


\begin{thebibliography}{abc}

\bibitem{a} D.D. Anderson, \emph{Star operations induced by overrings}, Comm. Algebra \textbf{16} (1988), 2535-2553.

\bibitem{aafz} D.D. Anderson, D.F. Anderson, M. Fontana and M. Zafrullah, \emph{On $v$-domains 
and star operations}, Comm. Algebra \textbf{37} (2009), 3018--3043.

\bibitem{aaj} D.D. Anderson, D.F. Anderson and E.W. Johnson, \emph{Some ideal theoretic
conditions on a Noetherian ring}, Houston J. Math. \textbf{7} (1981), 1-10.

\bibitem{aaz} D.D. Anderson, D.F. Anderson and M. Zafrullah, \emph{Completely integrally
closed Pr\"ufer $v$-multiplication domains}, Comm. Algebra, to appear.

\bibitem{ad} D.F. Anderson and D. Dobbs, \emph{Pairs of rings with the same prime ideals}, Canad. J. Math. \textbf{32} (1980), 362-384.

\bibitem{bz} A. Bouvier and M. Zafrullah, \emph{On some class groups of an integral domain}, Bull.
Soc. Math. Grece \textbf{29} (1988), 45-59.

\bibitem{ch} G.W. Chang, \emph{A new charaterization of Pr\"ufer v-multiplication domains}, 
Korean J. Math. \textbf{23} (2015), 631-636.

\bibitem{efz} S. El Baghdadi, M. Fontana and M. Zafrullah, \emph{Intersections of quotient
rings and Prufer $v$-multiplication domains}, J. Algebra Appl. \textbf{15} (2016), 165149, 18 pp.

\bibitem{fg} M. Fontana and S. Gabelli, \emph{On the class group and the local class group of a pullback}, J. Algebra \textbf{181} (1996), 803-835.

\bibitem{gh} S. Gabelli and E. Houston, \emph{Coherentlike conditions in pullbacks}, 
Michigan Math. J. \textbf{44} (1997), 99-123.

\bibitem{g} R. Gilmer, \emph{Multiplicative Ideal Theory}, Dekker, New York, 1972.

\bibitem{g2} R. Gilmer, \emph{Finite element factorization in group rings}, Lecture Notes in Pure and Appl. Math. \textbf{7}, Dekker, New York, 1974.

\bibitem{gv} S. Glaz and W. Vasconcelos, \emph{Flat ideals II}, Manuscripta Math. 
\textbf{22} (1977), 325-341.

\bibitem{h} W. Heinzer, \emph{An essential integral domain with a nonessential
localization}, Canad. J. Math. \textbf{33} (1981), 400--403.

\bibitem{hh} J. Hedstrom and E. Houston, \emph{Pseudo-valuation domains}, Pacific J. Math. \textbf{75} (1978), 137-147.

\bibitem{hz} E. Houston and M. Zafrullah, \emph{Integral domains in which each $t$-ideal
is divisorial}, Michigan Math. J. \textbf{35} (1988), 291-300.

\bibitem{hz2} E. Houston and M. Zafrullah, \emph{On $t$-invertibility II}, Comm. Algebra \textbf{17} (1989), 1955-1969.

\bibitem{k} B.G. Kang, {\em Pr\"ufer $v$-multiplication domains and the ring
$R[X]_{N_v}$}, J. Algebra {\bf 123} (1989), 151--170.

\bibitem{ll} A. Loper and T. Lucas, \emph{Factoring ideals in almost Dedekind domains}, J. Reine Angew Math. \textbf{565} (2003), 61-78.

\bibitem{u} H. Uda, \emph{LCM-stableness in ring extensions}, Hiroshima Math. J. \textbf{13} (1983), 357-377.

\bibitem{u2} H. Uda, \emph{$G_2$-stableness and LCM-stableness}, Hiroshima Math. J. \textbf{18}, (1988), 47-52.

\bibitem{wm} F. Wang and R. McCasland, \emph{On $w$-modules over strong Mori domains}, Comm. Algebra \textbf{25}  (1997), 1285-1306.

\bibitem{wm2} F. Wang and R. McCasland, \emph{On strong Mori domains}, J. Pure Appl. Algebra \textbf{135} (1999), 155-165.

\bibitem{zagcd} M. Zafrullah, \emph{A general theory of almost factoriality}, Manuscripta
Math. \textbf{51} (1985), 29-62.

\bibitem{zgdd} M. Zafrullah, \emph{On generalized Dedekind domains}, Mathematika, \textbf{33} (1986), 
285-295.

\bibitem{z} M. Zafrullah, \emph{On a property of pre-Schreier domains}, Comm. Algebra \textbf{15} (9)
(1987), 1895-1920.

\end{thebibliography}
\end{document}